\documentclass[a4paper,12pt]{article}

\usepackage[a4paper, top=3cm, left=2.5cm, right=2.5cm, bottom=3cm]{geometry}

\usepackage{amsmath, amsthm, pb-diagram}
\usepackage{amssymb}
\usepackage{amsfonts}
\usepackage{enumerate, fancyhdr, dsfont}

    \newcommand{\dom}{\mbox{\rm dom}}

    \newcommand{\thzfc}{\mathrm{ZFC}}

    \newcommand{\Ewf}{\mathcal{E}}

    \newcommand{\Iwf}{\mathcal{I}}
    
    \newcommand{\Mwf}{\mathcal{M}}

    \newcommand{\Pwf}{\mathcal{P}}

    \newcommand{\afrak}{\mathfrak{a}}
    \newcommand{\bfrak}{\mathfrak{b}}
    \newcommand{\cfrak}{\mathfrak{c}}
    \newcommand{\dfrak}{\mathfrak{d}}

    \newcommand{\sfrak}{\mathfrak{s}}

    \newcommand{\menos}{\smallsetminus}

    \newcommand{\frestr}{\!\!\upharpoonright\!\!}

    \newcommand{\cov}{\mbox{\rm cov}}
    \newcommand{\non}{\mbox{\rm non}}

    \newcommand{\Eor}{\mathds{E}}

    \newcommand{\Por}{\mathds{P}}
    \newcommand{\Qor}{\mathds{Q}}
    
    \newcommand{\Sor}{\mathds{S}}
    
    \newcommand{\Qnm}{\dot{\mathds{Q}}}
    
    \newcommand{\Snm}{\dot{\mathds{S}}}

    \newcommand{\R}{\mathbb{R}}

    \newcommand{\sii}{{\ \mbox{$\Leftrightarrow$} \ }}

\title{Borel computation of names in template iterations}
\author{Diego A. Mej\'ia\thanks{Supported by the Austrian Science Fund (FWF) P23875-N13 and I1272-N25}
}

\date{\small Institute of Discrete Mathematics and Geometry\\ Vienna University of Technology\\ Vienna, Austria.\\ \ \\ \texttt{diego.guzman@tuwien.ac.at}}

\begin{document}

\makeatletter
\def\@roman#1{\romannumeral #1}
\makeatother

\theoremstyle{plain}
  \newtheorem{theorem}{Theorem}[section]
  \newtheorem{corollary}[theorem]{Corollary}
  \newtheorem{lemma}[theorem]{Lemma}
  \newtheorem{prop}[theorem]{Proposition}
  \newtheorem{clm}[theorem]{Claim}
  \newtheorem{exer}[theorem]{Exercise}
  \newtheorem{question}[theorem]{Question}
\theoremstyle{definition}
  \newtheorem{definition}[theorem]{Definition}
  \newtheorem{example}[theorem]{Example}
  \newtheorem{remark}[theorem]{Remark}
  \newtheorem{notation}[theorem]{Notation}
  \newtheorem{context}[theorem]{Context}
  \newtheorem*{acknowledgements}{Acknowledgements}

\maketitle

\begin{abstract}
   We prove that, for a suitable iteration $\Por$ along a template $\langle L,\bar{\Iwf}\rangle$, we can compute any $\Por$-name for a real from a Borel function coded in the ground model evaluated at only countably many of the generic reals.
\end{abstract}

\section{Introduction}\label{SecIntro}

Consider the following class of definable ccc posets (see examples in Section \ref{SecBorelccc}).

\begin{definition}\label{DefcccBorel}
   A poset $\Sor$ is \emph{ccc Borel} if it is ccc, the relations $\leq_{\Sor}$ and $\perp_{\Sor}$ are Borel, it adds a (generic) real $\dot{\eta}$ and there is a Borel relation $E\subseteq\omega^\omega\times\omega^\omega$ such that,
   \begin{enumerate}[(i)]
      \item $E(z,\mathds{1}_{\Sor})$ is true for any real $z$ and
      \item in any $\Sor$-extension, $p\in\Sor$ is in the generic filter iff $E(\dot{\eta},p)$.
   \end{enumerate}
   A subposet $\Qor$ of $\Sor$ is \emph{nice} if $\Qor=\Sor^M$ for some transitive model $M$ of (a large fragment of) $\thzfc$ that contains $\omega_1$, $\dot{\eta}$ and the parameters of $\Sor$ and $E$.
\end{definition}

It is very common to use finite support iterations of nice subposets of Borel ccc posets (and also of quite small ccc posets) to obtain models where many cardinal invariants assume different values (see, for example, \cite{Br-Cichon}, \cite{JuSh-KunenMillerchart}, \cite{Me-MatIt} and \cite{Me-Matit02}). In \cite{Left Cichon}, the same technique is used to prove the consistency of $\bfrak<\non(\Mwf)<\cov(\Mwf)$ but, as it is hard to preserve unbounded families while using nice subposets of $\Eor$ (used to increase $\non(\Mwf)$, see Example \ref{ExpBorellinked}), new ideas like a construction of chains of ultrafilters had to be introduced to guarantee that $\leq^*$-increasing unbounded families in the ground model are preserved through the iteration. Here, it is necessary to code countable delta systems of conditions in the iteration without complete knowledge of what the iteration would be, that is, the code of these delta system can be interpreted once the iteration is constructed. This coding is possible because names of reals can be coded by Borel functions, as illustrated in the following fact.

\begin{theorem}[{\cite{Left Cichon}}]\label{fsiBorelComp}
   Let $\Por=\langle\Por_\alpha,\Qnm_\alpha\rangle_{\alpha<\delta}$ be a finite support iteration, $\delta=B\cup C$ disjoint union such that, for $\alpha\in B$, $\Qnm_\alpha$ is a $\Por_\alpha$-name of a nice subposet of a Borel ccc poset coded in the ground model and, for $\alpha\in C$, $\Qnm_\alpha$ is a $\Por_\alpha$-name of a ccc poset which domain, without loss of generality, is assumed to be an ordinal\footnote{In practice, these posets are small with respect to some fixed cardinal, this in order to have nice preservation properties for the iteration.}. If $\dot{x}$ is a $\Por$-name for a real, then there is a Borel function $F$ in the ground model such that $\Vdash\dot{x}=F(\langle\dot{\eta}_\alpha\rangle_{\alpha\in N})$ for some countable subset $N$ of $\delta$, where
   \begin{enumerate}[(i)]
      \item if $\alpha\in B\cap N$, $\dot{\eta}_\alpha$ is the name of the generic real added by $\Qnm_\alpha$ and
      \item if $\alpha\in C\cap N$, $\dot{\eta}_\alpha=\dot{\chi}_\alpha\frestr W_\alpha$ where $\dot{\chi}_\alpha$ is the characteristic function of the generic set added by $\Qnm_\alpha$ and $W_\alpha$ is a countable set, where $\langle W_\alpha\rangle_{\alpha\in C\cap N}$ belongs to the ground model.
   \end{enumerate}
\end{theorem}

The main objective of this text is to extend this coding of names by Borel functions to the context of iterations along a template. This is possible by considering template iterations that alternates between nice subposets of Borel $\sigma$-linked posets (some of them correctness-preserving, see Definition \ref{DefCorrPres}), coded in the ground model, and arbitrary $\sigma$-linked posets (which in practice, are quite small). We are going to call these \emph{simple template iterations} (see Definition \ref{DefSimpleiteration} for details). The main result is stated in detail in Theorem \ref{MainThm}.

The theory of template iterations was originally introduced Shelah \cite{Sh-TempIt} to construct a model of $\aleph_1<\dfrak<\afrak$. Further applications and generalizations of the template iteration theory are presented, for example, in \cite{Br-TempIt,Br-CtbleCof, Br-Luminy}, \cite{Me-TempIt}, \cite{MaxCofGr} and \cite{Fischer-Mejia}. Our notation about template iterations corresponds to \cite{Me-TempIt}.

\begin{acknowledgements}
   This paper was motivated from the talk ``$\sfrak\bfrak\afrak$" (joint work with V. Fischer \cite{Fischer-Mejia}) that the author contributed to the RIMS 2014 Workshop on Infinitary Combinatorics in Set Theory and Its Applications. The author is deeply thankful with T. Usuba for organizing such a wonderful conference.
\end{acknowledgements}

\section{Simple template iterations}\label{SecBorelccc}

In this section, we want to define the type of iterations we are interested in for the main result, which we call simple (template) iterations.

\begin{notation}\label{NotationNames}
   Given a ccc poset $\Por$, without loss of generality, we assume that any $\Por$-name $\dot{x}$ for a real is of the form $\langle h_n^{\dot{x}},A_n^{\dot{x}}\rangle_{n<\omega}$ where, for each $n<\omega$, $A_n=A_n^{\dot{x}}$ is a maximal antichain in $\Por$, $h_n=h_n^{\dot{x}}:A_n\to\omega$ and each $p\in A_n$ decides $\dot{x}(n)$ to be $h_n(p)$.
\end{notation}

\begin{lemma}\label{forceSimga1-2}
   Let $\Sor$ be a Suslin ccc poset. If $\varphi(z)$ is a $\boldsymbol{\Sigma}^1_1$-statement of reals and $\dot{x}$ is a $\Sor$-name for a real, then the statement ``$p\Vdash\varphi(\dot{x})$" is $\boldsymbol{\Sigma}^1_2$. On the other hand, if $\varphi(z)$ is a $\boldsymbol{\Pi}^1_1$-statement of reals, then ``$p\Vdash\varphi(\dot{x})$" is $\boldsymbol{\Pi}^1_2$.
\end{lemma}
\begin{proof}
   We first prove that, if $T\subseteq\omega^{<\omega}$ is a tree, then the statement ``$p\Vdash\dot{x}\in[T]$" is $\boldsymbol{\Sigma}^1_1\boldsymbol{\cup}\boldsymbol{\Pi}^1_1$ (the smallest $\sigma$-algebra containing both $\boldsymbol{\Sigma}^1_1$ and $\boldsymbol{\Pi}^1_1$). As in Notation \ref{NotationNames}, $\dot{x}=\langle h,A_n\rangle_{n<\omega}$ were $A_n=\{q_{n,i}\ /\ i<|A_n|\}$ is countable and $h_n:|A_n|\to\omega$ (in the sense that $q_{n,i}$ decides $\dot{x}(n)=h_n(i)$), so $\dot{x}$ can be seen as a real itself. Therefore, ``$\dot{x}$ is a $\Sor$-name for a real" is a $\boldsymbol{\Sigma}^1_1\boldsymbol{\cup}\boldsymbol{\Pi}^1_1$-statement (it is just $\boldsymbol{\Pi}^1_1$ if $\Sor$ is Borel ccc). Now, notice that $p\Vdash\dot{x}\frestr k\in T$ iff $p\in\Sor$, $\dot{x}$ is a $\Sor$-name for a real and, for every $s\in\omega^k$, if $\{q_{i,s(i)}\ /\ i<k\}\cup\{p\}$ has a common stronger condition in $\Sor$, then $\langle h_i(s(i))\rangle_{i<k}\in T$, which is a $\boldsymbol{\Sigma}^1_1\boldsymbol{\cup}\boldsymbol{\Pi}^1_1$-statement (or just $\boldsymbol{\Pi}^1_1$ if $\Sor$ is Borel).

   Recall that an analytic statement is the projection of $[T]$ for some tree $T\subseteq(\omega\times\omega)^\omega$. Note that $p\Vdash\exists_y((\dot{x},y)\in[T])$ iff $p\in\Sor$, $\dot{x}$ is a $\Sor$-name for a real and there is a $\Sor$-name for a real $\dot{y}$ such that $p\Vdash(\dot{x},\dot{y})\in[T]$, which is clearly a $\boldsymbol{\Sigma}^1_2$-statement.

   The other affirmation is proven similarly (because $p\Vdash\dot{x}\notin[T]$ is $\boldsymbol{\Pi}^1_2$).
\end{proof}

As a consequence of this Lemma we have that the generic filter of any nice subposet of a Borel ccc poset is also well described by the Borel relation of the Borel poset, as shown in the following result.

\begin{corollary}\label{Nicesubposet}
   Let $\Sor$ be a Borel ccc poset as in Definition \ref{DefcccBorel} and $\Qor$ a nice subposet of $\Sor$. If $G$ is $\Qor$-generic over $V$ and $p\in\Qor$, then $p\in G$ iff $E(\dot{\eta},p)$.
\end{corollary}
\begin{proof}
  $\Vdash p\in\dot{G}\sii E(\dot{\eta},p)$ is equivalent to say that $p\Vdash E(\dot{\eta},p)$ and, for every $q\in\Sor$, if $q\Vdash E(\dot{\eta},p)$ then $q\parallel p$, which is a $\boldsymbol{\Pi}^1_2$-statement by Lemma \ref{forceSimga1-2}. So $\forall_{p\in\Sor}(\Vdash p\in\dot{G}\sii E(\dot{\eta},p))$ is also $\boldsymbol{\Pi}^1_2$.

   Now, let $M$ a transitive model of (a large fragment of) $\thzfc$ that contains $\omega_1$, $\dot{\eta}$ and the parameters of $\Sor$ and $E$, such that $\Qor=\Sor^M$. By the absoluteness of $\boldsymbol{\Pi}^1_2$-statements, $M\models\forall_{p\in\Sor}(\Vdash p\in\dot{G}\sii E(\dot{\eta},p))$. If $G$ is $\Qor$-generic over $V$, then it is $\Qor$-generic over $M$, so $M[G]\models``p\in G\sii E(\eta[G],p)"$ for any $p\in\Qor$. Therefore, as $E$ is a Borel relation, the equivalence $``p\in G\sii E(\eta[G],p)"$ is also true in $V[G]$.
\end{proof}

Definable posets that are involved in simple iterations should satisfy the following two notions.

\begin{definition}[{\cite{Br-Luminy}}]\label{DefBorelsigmalinked}
   A poset $\Sor$ is \emph{Borel $\sigma$-linked} if it is Borel ccc (see Definition \ref{DefcccBorel}) and there is a sequence $\{S_n\}_{n<\omega}$ of linked sets such that the statement ``$x\in S_n$" is Borel. In addition, if all those $S_n$ are centered, we say that $\Sor$ is \emph{Borel $\sigma$-centered}.
\end{definition}

\begin{definition}[{\cite{Br-Luminy}}]\label{DefCorrPres}
   \begin{enumerate}[(1)]
      \item A system of posets $\langle\Por_0,\Por_1,\Qor_0,\Qor_1\rangle$ is \emph{correct} if $\Por_i$ is a complete subposet of $\Qor_i$ for $i=0,1$, $\Por_0$ is a complete subposet of $\Por_1$, $\Qor_0$ is a complete subposet of $\Qor_1$ and, whenever $p\in\Por_0$ is a reduction of $q\in\Qor_0$, then $p$ is a reduction of $q$ with respect to $\Por_1,\Qor_1$.
      \item A Suslin ccc poset $\Sor$ is \emph{correctness-preserving} if, for any $\langle\Por_0,\Por_1,\Qor_0,\Qor_1\rangle$ as in (1), the system $\langle\Por_0\ast\Snm^{V^{\Por_0}},\Por_1\ast\Snm^{V^{\Por_1}},
          \Qor_0\ast\Snm^{V^{\Qor_0}},\Qor_1\ast\Snm^{V^{\Qor_1}}\rangle$ is correct.
   \end{enumerate}
\end{definition}

\begin{example}\label{ExpBorellinked}
   \begin{enumerate}[(1)]
      \item Consider $\Eor$ the canonical forcing that adds an eventually different real, that is, conditions are of the form $(s,F)\in\omega^{<\omega}\times[\omega^\omega]^{<\omega}$ and the order is given by $(s',F')\leq(s,F)$ iff $s\subseteq s'$, $F\subseteq F'$ and $s(i)\neq x(i)$ for all $x\in F$ and $i\in|s'|\menos|s|$. It is clear that this poset has a Borel definition. $\dot{e}=\bigcup\{s\ /\ \exists_F((s,F)\in\dot{G})\}$ is the name of the generic real and, with the closed-relation $E(z,(s,F))$ defined as ``$s\subseteq z$ and $\forall_{i\geq|s|}\forall_{x\in F}(z(i)\neq x(i))$", it is clear that $\Vdash``(s,F)\in\dot{G}\sii E(\dot{e},(s,F))"$, so $\Eor$ is Borel ccc. It is also clear that $\Eor$ is Borel $\sigma$-centered.
      \item Classical forcing notions like Cohen forcing and Hechler forcing are Borel $\sigma$-centered, while localization forcing and random forcing are Borel $\sigma$-linked.
      \item All the previous posets are correctness-preserving, due to Brendle \cite{Br-Luminy,Br-Shat} (see also \cite[Sect. 2]{Me-TempIt}).
   \end{enumerate}
\end{example}

\begin{definition}\label{DefSimpleiteration}
   Let $\langle L,\bar{\Iwf}\rangle$ be an indexed template. A \emph{simple (template) iteration $\Por\frestr\langle L,\bar{\Iwf}\rangle$} consists of the following components:
   \begin{enumerate}[(i)]
      \item $L=B\cup R\cup C$ as a disjoint union.
      \item For $x\in B\cup R$ let $\Sor_x$ be a Borel $\sigma$-linked correctness-preserving poset, where $E_x$ is its corresponding Borel relation and $\dot{\eta}_x$ the name of its generic real.
      \item For $x\in R$  fix $C_x\in\hat{\Iwf}_x$.
      \item For $x\in C$ fix an ordinal $\gamma_x$ and $C_x\in\hat{\Iwf}_x$.
   \end{enumerate}
   For $x\in L$ and $A\in\hat{\Iwf}_x$, $\Qnm^A_x$ (the $\Por\frestr A$-name of the poset used at coordinate $x$ of the iteration) is defined as follows.
   \begin{enumerate}[(i)]
      \setcounter{enumi}{4}
      \item If $x\in B$ then $\Qnm^A_x=\Sor_x^{V^{\Por\upharpoonright A}}$.
      \item If $x\in R$, fix $\Qnm_x$ a $\Por\frestr C_x$-name of a nice subposet of $\Sor_x^{V^{\Por\upharpoonright C_x}}$. $\Qnm^A_x=\Qnm_x$ if $C_x\subseteq A$, or it is the trivial poset otherwise.
      \item If $x\in C$, fix $\Qnm_x$ a $\Por\frestr C_x$-name of a $\sigma$-linked poset with domain $\gamma_x$. $\Qnm^A_x=\Qnm_x$ if $C_x\subseteq A$, or it is the trivial poset otherwise.
   \end{enumerate}
   For $x\in C$, denote by $\dot{\eta}_x$ the $\Por\frestr(C_x\cup\{x\})$-name of the characteristic function of the generic subset of $\Qnm_x$. Besides, for $B\in\hat{\Iwf}_x$, $\dot{\eta}^B_x$ denotes the $\Por\frestr(B\cup\{x\})$-name of the generic subset of $\Qnm^B_x$, so $\dot{\eta}^B_x=\dot{\eta}_x$ if $C_x\subseteq B$ or $\dot{\eta}^B_x=\{(0,1)\}$ otherwise.

   If $A\subseteq L$, define $\Por^*\frestr A$ as the set of conditions $p\in\Por\frestr A$ such that, for $x\in C\cap\dom p$, $p(x)$ is an ordinal in $\gamma_x$ (not just a name).
\end{definition}

\begin{remark}\label{RemonSmallposets}
  \begin{enumerate}[(1)]
   \item We could just ignore the ordinal in (iv) and state in (vii) that $\Qnm$ is a $\Por\frestr C_x$-name for a $\sigma$-linked poset. This is because, by ccc-ness, we can find an ordinal $\gamma_x$ such that $\Por\frestr C_x$ forces that $\Qnm$ is densely embedded into a poset with domain $\gamma_x$. In practice, these ordinals are meant to be small (in \cite{Left Cichon}, assuming $\kappa=\bfrak=\cfrak$ in the ground model, ``small" means ``of size $<\kappa$").

   \item It is easy to see that, in a simple iteration as in Definition \ref{DefSimpleiteration}, $\Por^*\frestr A$ is dense in $\Por\frestr A$ for all $A\subseteq L$. By induction on $\mathrm{Dp}(A)$: let $p\in\Por\frestr A$, $x=\max(\dom p)$, so there exists an $A'\in\Iwf_x\frestr A$ such that $p\frestr L_x\in\Por\frestr A'$ and $p(x)$ is a $\Por\frestr A'$-name for a condition in $\Qnm_x^{A'}$. Assume $x\in C$. If $C_x\subseteq A'$, get $p'\leq p\frestr L_x$ in $\Por\frestr A'$ and some $\xi<\gamma_x$ such that $p'\Vdash p(x)=\xi$. Now, find $q'\leq p'$ in $\Por^*\frestr A'$ (by induction hypothesis), so $q=q'\cup\{(x,\xi)\}$ is in $\Por^*\frestr A$ and it is stronger that $p$. On the other hand, if $C_x\nsubseteq A'$, then $p(x)$ is the trivial condition, which can be assumed to be 0, so this case is handled like before. The case $x\in B\cup R$ is also similar (and simpler).

   \item In Definition \ref{DefSimpleiteration}, we can add more conditions to $\Por^*\frestr A$ depending on the posets used at coordinates $x\in B\cup R$. For example, for such an $x$ where $\Sor_x=\Eor$, we could further assume that, if $x\in\dom p$, then $p(x)=(s,\dot{F})$ where $s$ and $|\dot{F}|$ are already decided. Again, we obtain that $\Por^*\frestr A$ is dense in $\Por\frestr A$.
  \end{enumerate}
\end{remark}

\section{Borel computation}\label{SecBorelComp}

Throughout this section, fix an indexed template $\langle L,\bar{\Iwf}\rangle$ and a simple iteration $\Por\frestr\langle L,\bar{\Iwf}\rangle$ as in Definition \ref{DefSimpleiteration}.

\begin{definition}\label{DefHistory}
   By recursion on $\mathrm{Dp}(A)$, define, for any $p\in\Por^*\frestr A$ and $\dot{x}=\langle h_n,A_n\rangle_{n<\omega}$ a $\Por^*\frestr A$-name for a real:
   \begin{enumerate}[(1)]
     \item $H^A(p)\subseteq A$ and $H^A(\dot{x})\subseteq A$ as follows:
         \begin{enumerate}[(i)]
           \item If $x=\max(\dom p)$, choose $A'\in\Iwf_x\frestr A$ such that $p\frestr L_x\in\Por^*\frestr A'$ and $p(x)$ is a $\Por^*\frestr A'$-name for a condition in $\Qnm^{A'}_x$. Put $H^A(p)=H^{A'}(p\frestr L_x)\cup H^{A'}(p(x))\cup\{x\}$, but ignore $H^{A'}(p(x))$ when $x\in C$. On the other hand, if $p=\langle\ \rangle$, put $H^A(p)=\varnothing$.
           \item $H^A(\dot{x})=\bigcup\{H^A(p)\ /\ p\in A_n,\ n<\omega\}$.
         \end{enumerate}
     \item Sequences $\overline{W}^A(p)\in\prod_{z\in H^A(p)\cap C}\Pwf(\gamma_z)$ and $\overline{W}^A(\dot{x})\in\prod_{z\in H^A(\dot{x})\cap C}\Pwf(\gamma_z)$ as follows:
         \begin{enumerate}[(i)]
            \item If $x=\max(\dom p)$, choose $A'\in\Iwf_x\frestr A$ as in (1)(i) and put, for $z\in (H^{A'}(p\frestr L_x)\cup H^{A'}(p(x)))\cap C$, $W^A(p)_z=W^{A'}(p\frestr L_x)_z\cup W^{A'}(p(x))_z$ (ignore undefined terms in this union) and, if $x\in C$, put $W^A(p)_x=\{p(x)\}$ (recall that the trivial condition in $\Qnm_x^{A'}$ is 0 for $x\in C$).
            \item $W^A(\dot{x})_z=\bigcup\{W^A(p)_z\ /\ z\in H^A(p)\cap C,\ p\in A_n,\ n<\omega\}$ for $z\in H^A(\dot{x})\cap C$.
         \end{enumerate}
   \end{enumerate}
\end{definition}

It is necessary to see that both functions $H^A(\cdot)$ and $\overline{W}^A(\cdot)$ are well defined. Moreover, they do not depend on $A$, as follows from the following result.

\begin{lemma}\label{HistoryWellDef}
   Let $A'\subseteq A\subseteq L$.
   \begin{enumerate}[(a)]
      \item If $p\in\Por^*\frestr A'$ then $H^{A'}(p)=H^{A}(p)$ and $\overline{W}^{A'}(p)=\overline{W}^A(p)$.
      \item If $\dot{x}$ is a $\Por^*\frestr A'$-name for a real, then $H^{A'}(\dot{x})=H^{A}(\dot{x})$ and $\overline{W}^{A'}(\dot{x})=\overline{W}^A(\dot{x})$.
   \end{enumerate}
\end{lemma}
\begin{proof}
   We prove both (a) and (b) simultaneously by induction on $\mathrm{Dp}(A)$. Let $A'\subseteq A$ and $p\in\Por^*\frestr A'$. If $p=\langle\ \rangle$, clearly $H^{A'}(p)=H^{A}(p)$ and $\overline{W}^{A'}(p)=\overline{W}^A(p)$, so assume that $p\neq\langle\ \rangle$ and let $x=\max(\dom p)$. Then, there exists $K'\in\Iwf_x\frestr A'$ such that $p\frestr L_x\in\Por^*\frestr K'$ and $p(x)$ is a $\Por^*\frestr K'$-name for a condition in $\Qnm_x^{K'}$. Clearly, there is a $K\in\Iwf_x\frestr A$ containing $K'$ so, by induction hypothesis, $H^{K'}(p\frestr L_x)\cup H^{K'}(p(x))=H^K(p\frestr L_x)\cup H^K(p(x))$ and, for $z$ in this set and in $C$, $W^{K'}(p\frestr L_x)_z\cup W^{K'}(p(x))_z=W^K(p\frestr L_x)_z\cup W^K(p(x))_z$ (ignore undefined objects). Therefore, $H^{A'}(p)=H^A(p)$ and $\overline{W}^{A'}(p)=\overline{W}^A(p)$.

   If $\dot{x}$ is a $\Por^*\frestr A'$-name for a real, $H^{A'}(\dot{x})=H^A(\dot{x})$ and $\overline{W}^{A'}(\dot{x})=\overline{W}^A(\dot{x})$ follow straightforward.
\end{proof}

Lemma \ref{HistoryWellDef} allows us to denote $H(\cdot)=H^L(\cdot)$ and $\overline{W}(\cdot)=\overline{W}^L(\cdot)$. The intension of these two functions, which is materialized in Theorem \ref{MainThm} is that any $p\in\Por^*\frestr L$ can be reconstructed from the generic objects added at stages $x\in H(p)$ in the iteration and, for $z\in H(p)\cap C$, $p$ only depends on the information given by the set $W(p)_z$. Therefore, the same applies for $\Por^*\frestr L$-names for reals, which allows to define a Borel function in the ground model that determines $\dot{x}$ when it is evaluated at the generic reals from $H(\dot{x})$ where, for $z\in H(\dot{x})\cap C$, it is only needed to look at $W(\dot{x})_z$ intersected the generic set added at $z$. All these information from where conditions and names depend are countable, which is easily proved by induction on $\mathrm{Dp}(A)$ for $A\subseteq L$.

\begin{lemma}\label{CtbleHistory}
   For each $p\in\Por^*\frestr L$, $H(p)$ is a countable subset of $L$ and, for each $z\in H(p)\cap C$, $W(p)_z$ is a countable subset of $\gamma_z$. The same applies to $\Por^*\frestr L$-names for reals.
\end{lemma}

\begin{notation}\label{NotationPolish}
  Given a triple $\mathbf{t}=(H_S,H_C,\bar{W})$ where $H_S$ and $H_C$ are countable disjoint sets and $\bar{W}=\langle W_a\rangle_{a\in H_C}$ is a sequence of countable sets, define $\R(\mathbf{t})=(\omega^\omega)^{H_S}\times\prod_{a\in H_C}2^{W_a}$, which is clearly a Polish space. Additionally, for an arbitrary sequence $\bar{z}$ of functions, if $H_S\cup H_C\subseteq\dom(\bar{z})$, denote $\bar{z}\frestr\mathbf{t}=\langle z_a\rangle_{a\in H_S}\widehat{\ \ }\langle z_a\frestr W_a\rangle_{a\in H_C}$.
   
  Fix $A\subseteq L$, $p\in\Por^*\frestr A$ and $\dot{x}$ a $\Por^*\frestr A$-name for a real. For $p\in\Por^*\frestr A$, let $\mathbf{t}_p=(H(p)\menos C,H(p)\cap C,\bar{W}_p)$ and $\R(p):=\R(\mathbf{t}_p)$. Likewise, define $\mathbf{t}_{\dot{x}}$ and $\R(\dot{x})$.
  
  In particular, considering $\tilde{\eta}=\langle\dot{\eta}_z\rangle_{z\in L}$, $\tilde{\eta}\frestr\mathbf{t}_p$ and $\tilde{\eta}\frestr\mathbf{t}_{\dot{x}}$ are $\Por^*\frestr A$-names for reals in $\R(p)$ and $\R(\dot{x})$, respectively.
\end{notation}

We are now ready to state and prove the main result of this text.

\begin{theorem}\label{MainThm}
   Let $\Por\frestr\langle L,\bar{\Iwf}\rangle$ be a simple iteration as in Definition \ref{DefSimpleiteration}.
   \begin{enumerate}[(a)]
      \item There is a relation $\Ewf\subseteq\{(\bar{z},p)\ /\ p\in\Por^*\frestr L\textrm{\ and }\bar{z}\in\R(p)\}$ such that, for any $p\in\Por^*\frestr L$,
          \begin{enumerate}[(i)]
             \item $\Ewf(\cdot,p)$ is Borel in $\R(p)$ and
             \item $\Vdash_{\Por^*\upharpoonright L}p\in\dot{G}\sii\Ewf(\tilde{\eta}\frestr{\mathbf{t}_p},p)$.
          \end{enumerate}
      \item If $\dot{x}$ is a $\Por^*\frestr L$-name for a real, there exists a Borel function $F_{\dot{x}}:\R(\dot{x})\to\omega^\omega$ such that $\Vdash_{\Por^*\upharpoonright L}\dot{x}=F_{\dot{x}}(\tilde{\eta}\frestr\mathbf{t}_{\dot{x}})$.
   \end{enumerate}
\end{theorem}
\begin{proof}
   By recursion on $\mathrm{Dp}(A)$ for $A\subseteq L$, we define a relation $\Ewf^A\subseteq\{(\bar{z},p)\ /\ p\in\Por^*\frestr A\textrm{\ and }\bar{z}\in\R(p)\}$ such that, for any $p\in\Por^*\frestr A$,
   \begin{enumerate}[(i)]
       \item $\Ewf^A(\cdot,p)$ is Borel in $\R(p)$,
       \item $\Vdash_{\Por^*\upharpoonright A}p\in\dot{G}\sii\Ewf^A(\tilde{\eta}\frestr{\mathbf{t}_p},p)$ and
       \item for all $K\subseteq A$, $q\in\Por^*\frestr K$ and $\bar{z}\in\R(q)$, $\Ewf^K(\bar{z},q)$ iff $\Ewf^A(\bar{z},q)$.
   \end{enumerate}
   Within this recursion, for any $\Por^*\frestr A$-name for a real $\dot{x}$, we construct a Borel function $F^A_{\dot{x}}:\R(\dot{x})\to\omega^\omega$ such that
   \begin{enumerate}[(i)]
      \setcounter{enumi}{3}
      \item $\Vdash_{\Por^*\upharpoonright A}\dot{x}=F^A_{\dot{x}}(\tilde{\eta}\frestr\mathbf{t}_{\dot{x}})$ and
      \item for all $K\subseteq A$ and $\dot{y}$ $\Por^*\frestr K$-name for a real, $F^K_{\dot{y}}=F^A_{\dot{y}}$.
   \end{enumerate}
   This implies that $\Ewf^L$ and $F_{\dot{x}}=F^L_{\dot{x}}$ is as we want for (a) and (b).
   We proceed with the construction by the following cases.
   \begin{enumerate}[(1)]
      \item \emph{$A$ has a maximum $x$ and $A_x=A\cap L_x\in\hat{\Iwf}_x$.} We consider cases on $x$.
         \begin{itemize}
            \item[(1.1)] If $x\in B\cup R$, $\Ewf^A(\bar{z},p)$ iff $p\in\Por^*\frestr A$, $\bar{z}\in\R(p)$ and, either $x\notin\dom p$ and $\Ewf^{A_x}(\bar{z},p)$, or $x\in\dom p$, $\Ewf^{A_x}(\bar{z}\frestr\mathbf{t}_{p\upharpoonright L_x},p\frestr L_x)$ and $E_x(z_x,F^{A_x}_{p(x)}(\bar{z}\frestr\mathbf{t}_{p(x)}))$. Note that, when $x\in R$ and $p(x)$ is the trivial condition, $H(p(x))=\varnothing$ and $F^{A_x}_{p(x)}(\bar{z}\frestr\mathbf{t}_{p(x)})=\mathds{1}_{\Sor_x}$, so $E_x(z_x,F^{A_x}_{p(x)}(\bar{z}\frestr\mathbf{t}_{p(x)}))$ is true.
            \item[(1.2)] If $x\in C$, $\Ewf^A(\bar{z},p)$ iff $p\in\Por^*\frestr A$, $\bar{z}\in\R(p)$ and, either $x\notin\dom p$ and $\Ewf^{A_x}(\bar{z},p)$, or $x\in\dom p$, $\Ewf^{A_x}(\bar{z}\frestr\mathbf{t}_{p\upharpoonright L_x},p\frestr L_x)$ and $z_x(p(x))=1$.
         \end{itemize}
      \item \emph{$A$ has a maximum $x$ but $A_x\notin\hat{\Iwf}_x$.} $\Ewf^A(\bar{z},p)$ iff there is an $A'\subseteq A$ such that $A'\cap L_x\in\Iwf_x\frestr A$, $p\in\Por^*\frestr A'$, $\bar{z}\in\R(p)$ and $\Ewf^{A'}(\bar{z},p)$. By (iii), this is equivalent to say that, for any $A'\subseteq A$, if $A'\cap L_x\in\Iwf_x\frestr A$, $p\in\Por^*\frestr A'$ and $\bar{z}\in\R(p)$, then $\Ewf^{A'}(\bar{z},p)$.
      \item \emph{$A$ does not have a maximum.} Two cases
          \begin{itemize}
             \item[(3.1)] \emph{$A=\varnothing$.} $\Ewf^A(\bar{z},p)$ iff $p=\langle\ \rangle$ and $\bar{z}\in\R(p)=\{\langle\ \rangle\}$.
             \item[(3.2)] \emph{$A\neq\varnothing$.} $\Ewf^A(\bar{z},p)$ iff there are $x\in A$ and $A'\in\Iwf_x\frestr A$ such that $p\in\Por^*\frestr A'$, $\bar{z}\in\R(p)$ and $\Ewf^{A'}(\bar{z},p)$. By (iii), this is equivalent to say that, for any $x\in A$ and $A'\in\Iwf_x\frestr A$, if $p\in\Por^*\frestr A'$ and $\bar{z}\in\R(p)$ then $\Ewf^{A'}(\bar{z},p)$.
          \end{itemize}
   \end{enumerate}
   (i), (ii) and (iii) can be checked by simple calculations. We only show one case of (iii) to give an idea of how to proceed. Assume that $A$ is as in case (2) and $K\subseteq A$, also as in case (2), where $y=\max(K)\leq x$. Let $p\in\Por^*\frestr K$. We can find $A'\subseteq A$ and $K'\subseteq K\cap A'$ such that $K'\cap L_y\in\Iwf_y\frestr K$, $A'\cap L_x\in\Iwf_x\frestr A$ and $p\in\Por^*\frestr K'$ so, for $\bar{z}\in\R(p)$, $\Ewf^K(\bar{z},p)$ iff $\Ewf^{K'}(\bar{z},p)$ (by (2)) iff $\Ewf^{A'}(\bar{z},p)$ (by induction hypothesis) iff $\Ewf^A(\bar{z},p)$ (by (2)).
   
   Let $\dot{x}=\langle h_n,A_n\rangle_{n<\omega}$ as in Notation \ref{NotationNames}, where $A_n=\{p_{n,k}\ /\ k<\beta_n\}$ is a maximal antichain with $\beta_n=|A_n|$. Define $D\subseteq\R(\dot{x})$ such that $\bar{z}\in D$ iff, for all $n<\omega$, there is a unique $k<\beta_n$ such that $\Ewf^A(\bar{z}\frestr\mathbf{t}_{p_{n,k}},p_{n,k})$. Clearly, by (i), $D$ is Borel. Let $F:D\to\omega^\omega$ defined as $F(\bar{z})(n)=h_n(p_{n,k})$ where $k<\beta_n$ is the unique such that $\Ewf^A(\bar{z}\frestr\mathbf{t}_{p_{n,k}},p_{n,k})$. It is easy to see that $\Vdash_{\Por^*\upharpoonright A}``\tilde{\eta}\frestr\mathbf{t}_{\dot{x}}\in D\textrm{\ and }\dot{x}=F(\tilde{\eta}\frestr\mathbf{t}_{\dot{x}})"$ by (ii). In a trivial way, we can extend $F$ to a Borel function $F^A_{\dot{x}}$ with domain $\R(\dot{x})$. (v) is an immediate consequence of (iii).
\end{proof}

{\small

\bibliographystyle{spmpsci}      

\begin{thebibliography}{BML}

\bibitem[B91]{Br-Cichon}
J. Brendle:
\emph{Larger cardinals in Cicho\'n's diagram,}
J. Symb. Logic 56, no. 3 (1991) 795-810.

\bibitem[B02]{Br-TempIt}
J. Brendle:
\emph{Mad families and iteration theory.}
In: Logic and Algebra, Y. Zhang (ed.), Contemp. Math. 302, Amer. Math. Soc., Providence, RI, 2002, pp. 1-31.

\bibitem[B03]{Br-CtbleCof}
J. Brendle:
\emph{The almost disjointness number may have countable cofinality.}
Trans. Amer. Math. Soc. 355 (2003) 2633-2649.


\bibitem[B05]{Br-Luminy}
J. Brendle:
\emph{Templates and iterations, Luminy 2002 lecture notes.}
Ky\={o}to daigaku s\={u}rikaiseki kenky\={u}sho k\={o}ky\={u}roku (2005) 1-12.

\bibitem[B]{Br-Shat}
J. Brendle:
\emph{Shattered iterations.}
In preparation.

\bibitem[FT]{MaxCofGr}
V. Fischer, A. T\"ornquist:
\emph{Template iterations and maximal cofinitary groups.}
Fund. Math., to appear.

\bibitem[FM]{Fischer-Mejia}
V. Fischer, D.A. Mej\'ia:
\emph{Splitting, bounding and almost disjointness can be quite different.}
In preparation.

\bibitem[GMS]{Left Cichon}
M. Goldstern, D. A. Mej\'ia, S. Shelah:
\emph{The left hand side of Cicho\'n's diagram.}
In preparation.

\bibitem[JS90]{JuSh-KunenMillerchart}
H. Judah, S. Shelah:
\emph{The Kunen-Miller chart (Lebesgue measure, the Baire property, Laver reals and preservation theorems for forcing).}
J. Symb. Logic 55, no. 3 (1990) 909-927.


\bibitem[M13a]{Me-MatIt}
D. A. Mej\'ia:
\emph{Matrix iterations and Cicho\'n's diagram,}
Arch. Math. Logic 52 (2013) 261-278.

\bibitem[M13b]{Me-Matit02}
D. A. Mej\'ia:
\emph{Models of some cardinal invariants with large continuum.}
Ky\={o}to daigaku s\={u}rikaiseki kenky\={u}sho k\={o}ky\={u}roku (2013) 36-48.

\bibitem[M]{Me-TempIt}
D. A. Mej\'ia:
\emph{Template iterations with non-definable ccc forcing notions.}
Submitted.

\bibitem[S04]{Sh-TempIt}
S. Shelah:
\emph{Two cardinal invariants of the continuum ($\dfrak<\afrak$) and FS linearly ordered iterated forcing.}
Acta Math. 192 (2004) 187-223 (publication number 700).


\end{thebibliography}


}
\end{document}